\documentclass[12pt]{amsart}
\usepackage{amssymb,amsmath,amsthm,mathrsfs}
\numberwithin{equation}{section}
\usepackage{fullpage}
\usepackage{enumerate}
\usepackage{hyperref}

\newtheorem{thm}{Theorem}[section]
\newtheorem{lem}[thm]{Lemma}

\newtheorem{cor}[thm]{Corollary}
\newtheorem{prop}[thm]{Proposition}
\theoremstyle{definition}

\newtheorem{defn}[thm]{Definition}

\newcommand{\be}{\begin{equation}}
\newcommand{\ee}{\end{equation}}
\newcommand{\ba}{\begin{eqnarray}}
\newcommand{\ea}{\end{eqnarray}}
\newcommand{\bi}{\begin{itemize}}
\newcommand{\ei}{\end{itemize}}
\newcommand{\bn}{\begin{enumerate}}
\newcommand{\en}{\end{enumerate}}
\newcommand{\bbm}{\begin{bmatrix}}
\newcommand{\ebm}{\end{bmatrix}}
\newcommand{\bp}{\begin{proof}}
\newcommand{\ep}{\end{proof}}
\newcommand{\nn}{\nonumber}

\newcommand{\C}{\mathbb{C}}

\newcommand{\T}{\mathbb{T}}

\newcommand{\B}{\mathbb{B}}
\renewcommand{\L}{\mathcal{L}}

\newcommand{\gd}{{G_\delta}}
\newcommand{\cd}{{C_\delta}}

\newcommand{\mc}{\ensuremath{\mathcal}}

\newcommand{\ov}{\ensuremath{\overline}}

\newcommand{\ip}[2]{\ensuremath{\langle {#1} , {#2} \rangle}}

\begin{document}
\title[Extremal multipliers of the DA space]{Extremal multipliers of the Drury-Arveson space}

\author{M.T. Jury}
\address{University of Florida}
\email{mjury@ad.ufl.edu}

\author{R.T.W. Martin}
\address{University of Cape Town}
\email{rtwmartin@gmail.com}

\thanks{Second author acknowledges support of NRF CPRR Grant 90551.}
\begin{abstract}
We give a new characterization of the so-called {\em quasi-extreme} multipliers of the Drury-Arveson space $H^2_d$, and show that every quasi-extreme multiplier is an extreme point of the unit ball of the multiplier algebra of $H^2_d$.
\end{abstract}

\maketitle

\section{Introduction}
In \cite{Jur1} and \cite{JurMar} we introduced the notion of a {\em quasi-extreme} multiplier of the Drury-Arveson space $H^2_d$, and gave a number of equivalent formulations of this property. (The relevant definitions are recalled in Section~\ref{sec:review}.) The main motivation is that in one variable, each of these conditions is equivalent to $b$ being an extreme point of the unit ball of $H^\infty$ (the space of bounded analytic functions in the unit disk $\mathbb D\subset \mathbb C$). The purpose of this paper is to give one further characterization of quasi-extremity in the general case, from which it will follow that every quasi-extreme multiplier of $H^2_d$ is in fact an extreme point of the unit ball of the multiplier algebra $\mathcal{M}(H^2_d)$. (The converse statement, namely whether or not every extreme point is quasi-extreme in our sense, remains open.)  In particular we will prove the following theorem:
\begin{thm}\label{thm-main}
  A contractive mulitplier $b$ of $H^2_d$ is quasi-extreme if and only if the only multiplier $a$ satisfying
  \begin{equation}\label{}
M_a^*M_a +M_b^*M_b\leq I
  \end{equation}
is $a\equiv 0$.
\end{thm}

Since the corollary concerning extreme points follows immediately, we prove it here:

\begin{cor}
  If $b\in ball(\mathcal{M}(H^2_d))$ is quasi-extreme, then $b$ is an extreme point of $ball(\mathcal{M}(H^2_d))$.
\end{cor}
\begin{proof} We prove the contrapositive.  If $b$ is not extreme, then there exists a nonzero $a\in ball(\mathcal{M}(H^2_d))$ such that both $b\pm a$ lie in $ball(\mathcal{M}(H^2_d))$, that is, are contractive multipliers of $H^2_d$. We then have the operator inequalities
  \begin{equation*}
    M_{b+a}^*M_{b+a} \leq I, \quad M_{b-a}^*M_{b-a} \leq I
  \end{equation*}
averaging these inequalities gives
\begin{equation*}
  M_a^*M_a +M_b^*M_b\leq I
\end{equation*}
so by Theorem~\ref{thm-main}, $b$ is not quasi-extreme.
\end{proof}

The remainder of the paper is devoted to proving Theorem~\ref{thm-main}. Since the techniques required are rather different, the two implications of the theorem will be proved as two separate propositions, Propositions~\ref{prop-a-function} and \ref{minorant-implies-notqe}.  In Section~\ref{sec:review} we recall the Drury-Arveson space and its multipliers, and review the necessary results concerning the de-Branges Rovnyak type spaces $\mathcal H(b)$ conctractively contained in $H^2_d$, and in particular the solutions to the Gleason problem in these spaces. We define the {\em quasi-extreme} multipliers and review some equivalent formulations of this property that will be used later. In Section~\ref{sec:non-qe} we study the non-quasi-extreme multipliers in more detail, and extend to this class of functions some of the results proved by Sarason \cite{Sar} in the one-variable case. Our results rely heavily on the the construction of a particular solution to the Gleason problem with good extremal properties, which is carried out in this section. In Sections~\ref{sec:a-function} and \ref{sec:forward} respectively we prove Propositions~\ref{prop-a-function} and \ref{minorant-implies-notqe}.

\section{The Drury-Arveson space, multipliers, and
  quasi-extremity}\label{sec:review}

The {\em Drury-Arveson space} is the Hilbert space of holomorphic
functions defined on the unit ball $\mathbb B^d\subset \mathbb C^d$
with reproducing kernel
\begin{equation*}
  k_w(z)=k(z,w)=\frac{1}{1-zw^*}; \quad \quad z,w \in \B ^d
\end{equation*}
(Here we use the notation $z=(z_1 , z_2, \dots, z_d)$, so that $zw^* =
\sum_{j=1}^d z_j\overline{w_j}$.)  General facts about the $H^2_d$ spaces may be found in the recent survey \cite{Sha}.

A holomorphic function $b$ on $\mathbb B^d$ will be called a {\em multiplier} if $bf\in H^2_d$ whenever $f\in H^2_d$. In this case the operator $M_b:f\to bf$ is bounded, and we let $\mathcal{M}(H^2_d)$ denote the Banach algebra of multipliers, equipped with the operator norm. (Warning: the multiplier norm always dominates the supremum norm of $b$ over the unit ball, but the two are in general unequal. Also, not every bounded function $b$ is a multiplier, see \emph{e.g.} \cite{Sha,ArvI}.)  For the reproducing kernel $k_w$ we have $M_b^*k_w =b(w)^* k_w$. It follows that $\|M_b\|\leq 1$ if and only if the expression
\begin{equation*}
 k^b_w(z)=k^b(z,w):= \frac{1-b(z)b(w)^* }{1-zw^*},
\end{equation*}
defines a positive kernel on $\mathbb B^d$. When this is the case we let $\mathcal H(b)$ denote the associated reproducing kernel Hilbert space, called the deBranges-Rovnyak space of $b$. The space $\mathcal H(b)$ is a space of holomorphic functions on $\mathbb B^d$, contained in $H^2_d$, and the inclusion map $\mathcal H(b)\subset H^2_d$ is contractive for the respective Hilbert space norms. We write $\|\cdot\|_b$ and $\langle \cdot, \cdot\rangle_b$ for the norm and inner product in $\mathcal H(b)$ respectively.

Properties of the spaces $\mathcal H(b)$ when $d>1$ were studied in \cite{Jur1}, inspired among other things by the results of Sarason in the one variable case \cite{Sar},\cite{Sar-book}. In one variable, the $\mathcal H(b)$ spaces are invariant under the backward shift; in several variables we instead (following Ball, Bolotnikov, and Fang \cite{BBF}) consider solutions to the {\em Gleason problem}: given a function $f\in\mathcal H(b)$, we seek functions $f_1, \dots f_d\in \mathcal H(b)$ such that
\begin{equation}\label{eqn:gleason-prob}
  f(z)-f(0)=\sum_{j=1}^d z_j f_j(z).
\end{equation}
From \cite{BBF} we know that this problem always has a solution; in fact there exist (not necessarily unique) bounded operators $X_1, \dots X_d$ acting on $\mathcal H(b)$ such that the functions $f_j:=X_jf$ solve (\ref{eqn:gleason-prob} for any $f \in \mc{H}(b)$. Moreover these $X_j$ can be chosen to be {\em contractive} in the following sense: for every $f\in\mathcal H(b)$,
\begin{equation}
\sum_{j=1}^d \|X_jf\|^2_b\leq \|f\|_b^2 -|f(0)|^2.
\end{equation}
These contractive solutions were studied further in \cite{Jur1}, where we proved the following (see also \cite{JurMar} for the vector-valued case):

\begin{prop}\label{prop:gleason} A set of bounded operators $(X_1, \dots X_d)$ is a contractive solution to the Gleason problem in $\mathcal H(b)$ if and only if the $X_j$ act on reproducing kernels by the formula
\begin{equation}\label{Xj-def}
  X_jk_w^b = w_j^*k_w^b - b(w)^*b_j
\end{equation}
for some choice of functions $b_1, \dots b_d\in \mathcal H(b)$ which satisfy
\begin{itemize}
\item[(i)] $\sum_{j=1}^d z_j b_j(z) = b(z)-b(0)$,
\item[(ii)] $\sum_{j=1}^d\|b_j\|^2_b \leq 1-|b(0)|^2$.
\end{itemize}

\end{prop}
The set all contractive solutions $X$ is in one-to-one correspondence with the set of all tuples $b_1, \dots b_d$ satisfying these conditions \cite[Theorem 4.10]{JurMar}. We will call such sets of $b_j$ \emph{admissable}, or say that such a set is a \emph{contractive Gleason solution for $b$}.

In turns out that for some contractive multipliers $b$, the operators $X_j$ of the proposition are unique, that is, there is only one admissible tuple. When this happens we will call the multiplier $b$ {\em quasi-extreme}. (The original definition of quasi-extreme in \cite{Jur1} is different, involving the so-called noncommutative Aleksandrov-Clark state for $b$, but this definition will be easier to work with for the present purposes.) In \cite{Jur1} and \cite{JurMar} we gave a number of equivalent formulations of quasi-extremity, we recall only a few of them here.
\begin{prop}\label{prop:qe-conditions} Let $b$ be a contractive multiplier of $H^2_d$. The following are equivalent:
  \begin{itemize}
  \item[i)] $b$ is quasi-extreme.
\item[ii)] There is a unique contractive solution $(X_1, \dots X_d)$ to the Gleason problem in $\mathcal H(b)$.
\item[iii)] There exists a contractive solution $(X_1, \dots X_d)$ such that the equality $\sum_{j=1}^d \|X_jf\|_b^2 =\|f\|^2_b -|f(0)|^2$ holds for every $f\in\mathcal H(b)$.
\item[iv)] There is a unique admissible tuple $(b_1, \dots b_d)$ satisfying the conditions of Proposition~\ref{prop:gleason}.
\item[v)] All admissible tuples $(b_1, \dots b_d)$ are \emph{extremal}, \emph{i.e.}
$$\sum_{j=1}^d \|b_j\|_b^2 =1-|b(0)|^2,$$ for any admissable tuple.
\item[vi)] $\mathcal H(b)$ does not contain the function $b$.
\item[vii)] $\mathcal H(b)$ does not contain the constant functions.
\end{itemize}

\end{prop}
In \cite{JurMar} these equivalences were extended to the case of operator-valued $b$.

What will be most useful in what follows is item (v); in particular $b$ is {\em not} quasi-extreme if and only if there exists an admissible tuple $(b_1, \dots b_d)$ which obeys the strict inequality
\begin{equation*}
  \sum_{j=1}^d \|b_j\|_b^2 <1-|b(0)|^2.
\end{equation*}
\section{Non-quasi-extreme $b$}\label{sec:non-qe}

Let $b$ denote a contractive multiplier of the Drury-Arveson space $H^2_d$ on the unit ball $\mathbb B^d\subset \mathbb C^d$.  We assume throughout that $b(z)$ is not constant. We let $G_b(z)$ denote the Cayley transform or Herglotz-Schur function of $b$:
\begin{equation}
  G_b(z)=\frac{1+b(z)}{1-b(z)}
\end{equation}
and we contstruct the reproducing kernel Hilbert spaces  $\mathcal H(b), \mathscr L(b)$, the deBranges-Rovnyak and Herglotz spaces of $b$, respectively, with the kernels
\begin{equation}
  k_w^b(z) := \frac{1-b(z)b(w)^*}{1-zw^*},
\end{equation}
and,
\begin{equation}\label{L(b)-kernel}
  K_w^b(z) := \frac{1}{2} \frac{G_b(z)+G_b(w)^*}{1-zw^*} = (1-b(z))^{-1} (1-b(w)^*)^{-1} k^b_w(z).
\end{equation}
 The map $f\to (1-b)f$ is thus a unitary multiplier from $\mathscr L(b)$ onto $\mathcal H(b)$.

We define an operator $V:\mathscr L(b)^d\to \mathscr L(b)$ by declaring
\begin{equation}
  V\begin{pmatrix} z_1^* K_z^b \\ \vdots \\ z_d^* K_z^b\end{pmatrix} := K_z^b -K_0^b
\end{equation}
on the span of the the columns appearing in the definition, and defining $V$ to be $0$ on the orthogonal complement of this span. A quick calculation using the formula for the reproducing kernel (\ref{L(b)-kernel}) shows that
\begin{equation*}
  zw^* K^b_w(z) =\langle K^b_w-K_0^b,K_z^b-K_0^b\rangle_{\mathscr L(b)}\quad \text{ for all } z,w\in\mathbb B^d
\end{equation*}
and hence that $V$ is a partial isometry.  It follows that $V^*$ is $0$ on the orthogonal complement of the set $\{K_z^b-K_0^b:z\in\mathbb B^d\}$.  Note that a vector $f\in \mathscr L(b)$ is orthogonal to this set if and only if $f(z)=f(0)$ for all $z$; that is, if and only if $f$ is constant. (In particular $V$ is a coisometry if and only if the only constant function in $\mathscr L(b)$ is $0$.)  We next observe:
\begin{lem}
  The space $\mathscr L(b)$ contains the constants if and only if $\mathcal H(b)$ contains $b$; that is, by Proposition~\ref{prop:qe-conditions}, if and only if $b$ is not quasi-extreme.
\end{lem}
\begin{proof}
  $1\in\mathscr L(b)$ if and only if $1-b\in \mathcal H(b)$; since $k_0^b=1-b(z)b(0)^* \in\mathcal H(b)$ always, the lemma follows.
\end{proof}
For the remainder of this section we assume that $b$ is not
quasi-extreme, so by the lemma, $\mathscr L(b)$ contains the constant functions.  By construction the tuple $(V_1, \dots V_d)$ is a row
contraction and $\sum_{j=1}^d V_j V_j^* $ is the projection in
$\mathscr L(b)$ orthogonal to the constants; so that $V_j^*1=0$ for all
$j$.  We first record some facts about the $V_j$ that will be of use
later.

From the definition of $V$ we have
\begin{equation}
  \label{eq:1}
  V_j^*(K_z-K_0) = z_j^* K_z^b.
\end{equation}
We also record the following chain of equalities for later use; these
use only the fact that $V_j^*1=0$: For each $j$,
\ba \label{Vj*K0}
V_j^*(K_0) & = &   V_j^* \frac{1}{2} \left( \frac{2}{1-b} -1 + \frac{1+ b(0) ^* }{1-b(0)} ^* \right) \nn \\
 & = &  V_j ^* \left( \frac{1}{1-b}-1 \right) \nn \\
 & = &V_j^* \left( \frac{b}{1-b} \right). \label{Vj*K0} \ea

We next observe that the $V_j$ solve the Gleason problem in $\mathscr L(b)$; indeed for $f\in\mathscr L(b)$ we take the inner product of $f$ with the identity
\begin{equation}
  K_z^b -K_0^b = \sum_{j=1}^d z_j^*V_j K_z^b
\end{equation}
we get
\begin{equation}
  f(z)-f(0)= \sum_{j=1}^d z_j (V_j^*f) (z).
\end{equation}

We can now define operators $S_j$ on $\mathcal H(b)$ conjugate to the $V_j$ via the unitary $g\to \frac{1}{\sqrt{2}}(1-b)g$; specifically for $g\in\mathcal H(b)$ we define
\begin{equation}
  (S_j^*g)(z) = (1-b)V_j^*\frac{g}{1-b}.
\end{equation}
Again the row $S=(S_1, \dots, S_d)$ is a row contraction; in fact a row partial isometry whose final space $\text{ran}(\sum_{j=1}^d S_j S_j^*)$ is the orthognal complement of the one-dimensional space spanned by $1-b$.

We now use the operators $S_j$ to define an admissible tuple $b_1, \dots b_d$ and construct the associated solution to the Gleason problem in $\mathcal H(b)$ with good extremal properties. In particular put
\begin{equation}\label{bj-def}
  b_j = (1-b(0) )S_j^*b \in\mathcal H(b)
\end{equation}
and define operators $X_j$ as in (\ref{Xj-def}).

\begin{prop}\label{canonical-solution}
  The tuple $X=(X_1, \dots X_d)$ is a contractive solution to the Gleason problem in $\mathcal H(b)$. Moreover it is the unique solution with the property such that
  \begin{equation}
    X_jb = b_j
  \end{equation}
where the $b_j$ are those belonging to $X_j$.
\end{prop}
\begin{proof}
    The fact that $V^b$ can be used to define a contractive Gleason solution for $K(b)$ in this way is a special case of \cite[Theorem 4.4, Lemma 4.6]{JurMar}. We include a proof below for completeness.

  We first verify that the $b_j$ defined by (\ref{bj-def}) are
  admissible. Sine the $V_j^*$ solve the Gleason problem in $\mathscr
  L(b)$, we have
  \begin{align*}
    \sum_{j=1}^d z_j V_j^*(\frac{b}{1-b}) &= \frac{b(z)}{1-b(z)} -
    \frac{b(0)}{1-b(0)} \\
&= \frac{b(z)-b(0)}{(1-b(z))(1-b(0))}
  \end{align*}
so by the definition of $S_j^*$ and $b_j$
\begin{align*}
  b(z)-b(0) &= (1-b(0))(1-b(z))\sum_{j=1}^d z_j (V_j^*\frac{b}{1-b})(z) \\
&= (1-b(0))\sum_{j=1}^d z_j S_j^*b(z)\\
&= \sum_{j=1}^d z_jb_j(z).
\end{align*}
To prove the norm inequality, observe that
\begin{align*}
  \sum_{j=1}^d\|b_j\|^2_b &= \sum_{j=1}^d\|S_j^*b\|^2_b \\
&= |1-b(0)|^2 \sum_{j=1}^d \|V_j^*\frac{b}{1-b}\|^2_{\mathscr L(b)} \\
&= |1-b(0)|^2 \sum_{j=1}^d\|V_j^*K_0\|^2_{\mathscr L(b)}  \quad \text{by (\ref{Vj*K0})}\\
&\leq 1-|b(0)|^2.
\end{align*}
where the last inequality holds since $V^*$ is a column contraction
and $\|K_0||^2 = \frac{1-|b(0)|^2}{|1-b(0)|^2}$. Moreover, we observe that, since $V^*$ is a partial isometry, equality holds in the above chain if and
only if $K_0$ is orthogonal to the scalars, but this obviously never
happens, so the inequality is always strict in this case when $V^b$ is not a co-isometry. This also shows that this choice of admissible $b_j$ minimizes the sum $\sum_{j=1}^d \|b_j\|_b$ over all choices of admissible $b_j$ (see also \cite[Corollary 4.8, Remark 4.9]{JurMar}).

To show that $X_jb=b_j$, we first show that
\begin{equation}\label{Xj-vs-Sj}
  X_j = S_j^* - (1-b(0))^{-1}b_j\otimes k_0^b.
\end{equation}
This equation follows from Clark-type intertwining formulas of \cite[Section 4.15]{JurMar}.

Indeed, from (\ref{Vj*K0}) we have
\begin{equation*}
  V_j^*K_0 = V_j^*(\frac{1}{1-b}) = \frac{1}{1-b} S_j^*b =
  \frac{1}{(1-b)(1-b(0))} b_j
\end{equation*}
The formula (\ref{Xj-vs-Sj}) is then verified by checking it on kernels $k_w^b$, where we
have from the definition of $S_j^*$
\begin{align*}
  S_j^*k_w^b &= (1-b) V_j^*(\frac{1-bb(w)^*}{(1-b)(1-zw^*)}) \\
&=(1-b)(1-b(w)^*)V_j^*(K_w^b) \\
&= (1-b)(1-b(w)^*) V_j^*(K_w-K_0 +K_0)\\
&= w_j^*k_w^b + \frac{1-b(w)^*}{1-b(0)} b_j
\end{align*}
and so
\begin{align*}
  S_j^*k_w^b - [(1-b(0))^{-1}b_j\otimes k_0^b]k_w^b &= w_j^*k_w^b + \frac{1-b(w)^*}{1-b(0)} b_j  -\frac{1}{1-b(0)}(1-b(0)b(w)^*)b_j \\
&= w_j^*k_w^b-b(w)^*b_j \\
&= X_jk_w^b
\end{align*}
as desired.

Finally, the claim that $X_jb=b_j$ follows immediately from (\ref{Xj-vs-Sj}) and the definition of the $b_j$ in (\ref{bj-def}).
\end{proof}

{\em Remark:} We observe in passing that these $X_j$ annihilate the scalars: indeed, from the definition of $X_j$ in (\ref{Xj-def}) and the fact that $X_jb=b_j$, we have
\begin{equation*}
  X_j1 = X_j(1-b(0)^*b+b_0^*b) = X_j k_0^b + b(0)^*X_jb= -b(0)^*b_j+b(0)^*b_j=0.
\end{equation*}

We also have that the defect operator $I-\sum X_j^*X_j$ has rank two
when $b$ is non-extreme:
\begin{prop}\label{defect-identity}
  Let $b$ be a non-extreme multiplier. If $X_j$ is a solution to the
  Gleason problem in $\mathcal H(b)$ with $\sum\|b_j\|^2 = 1-|b(0)|^2
  -|a_0|^2$, then
  \begin{equation}\label{eqn:defects}
    I-\sum X_j^*X_j = k_0^b\otimes k_0^b +|a_0|^2 b\otimes b.
  \end{equation}
\end{prop}
\begin{proof}
  We first compute the inner product $\langle X_j^* X_j k_w^b, k_z^b\rangle$, using (\ref{Xj-def}):
  \begin{align*}
    \langle X_j^* X_j k_w^b, k_z^b\rangle &= \langle X_j k_w^b, X_j k_z^b \rangle \\
 &= \langle w_j^*k_w^b - b(w)^*b_j, z_j^*k_z^b - b(z)^*b_j \rangle \\
&= z_j w_j^*k^b(z,w) -z_jb_j(z)b(w)^* -w_j^*b_j(w)^* b(z) + \|b_j\|^2_bb(z)b(w)^*.
  \end{align*}
Summing over $j=1, \dots d$ (and using the fact that the $b_j$ are admissible) gives
\begin{equation*}
  \sum_{j=1}^d \langle X_j^* X_j k_w^b, k_z^b\rangle = zw^* k^b(z,w) -(b(z)-b(0))b(w)^* -b(z)(b(w)^*-b(0)^*) +(1-|b(0)|^2 -|a_0|^2)b(z)b(w)^*.
\end{equation*}
Finally, we find
\ba
\ip{(I -\sum_{j=1}^d X_j X_j^*)k_w^b}{ k_z^b } & = & (1-zw^*) k^b(z,w) +(b(z)-b(0))b(w)^* +b(z)(b(w)^*-b(0)^*) \nn \\
&  & -(1-|b(0)|^2 -|a_0|^2)b(z)b(w)^*  \nn \\
&=&  1- b(z)b(0)^* -b(0)b(w)^* +b(0)b(z)b(w)^*b(0)^* +|a_0|^2 b(z)b(w)^* \nn \\
&=& \ip{(k_0^b\otimes k_0^b + |a_0|^2 b\otimes b)k_w^b}{k_z^b}. \nn \ea
This completes the proof.
\end{proof}
Since $b$ is assumed non-constant, this proposition shows that the range of $I-\sum_{j=1}^d X_j^*X_j$ is two dimensional, spanned by $k_0=1-b(0)^*b$ and $b$, or equivalently, by $b$ and $1$.  We can use this fact and the foregoing identity to relate the ``defect'' $|a_0|^2:=1-|b(0)|^2-\sum_{j=1}^d \|b_j\|_b^2$ of the admissible tuple $b_1, \dots b_d$ to the $\mathcal H(b)$-norm of the function $b$ (compare \cite{Sar},\cite{BalKri} for the scalar and vector valued cases, respectively, in one variable).

\begin{lem}\label{a0-calculation} Suppose $b$ is not quasi-extreme, and let $(X_1, \dots X_j)$ and $b_1, \dots b_j$ be as in Proposition~\ref{canonical-solution}. Define $a_0>0$ by $|a_0|^2 = 1-|b(0)|^2-\sum_{j=1}^d \|b_j\|^2_b$. Then
  \begin{equation}\label{a0-formula}
    |a_0|^2 = \frac{1}{1+\|b\|_b^2}.
  \end{equation}
\end{lem}
\begin{proof}
  We compute $(I-\sum_{j=1}^d X_j^*X_j)b$ in two different ways. First, from its definition, and using the fact that $X_jb=b_j$,
  $$ (I-\sum_{j=1}^d X_j^*X_j)b = b - \sum _{j=1} ^d X_j ^* b_j. $$
  Then observe that
  \ba \sum _{j=1} ^d (X_j ^* b_j) (z) & = & \sum _{j=1} ^d \ip{ X^* _j b_j}{k_z ^b} _b \nn \\
  & = & \sum _{j=1} ^d \ip{b_j}{z_j ^* k_z ^b - b_j b(z) ^* } _b \nn \\
  &= & b(z) - b(0) - \sum _{j=1} ^d \| b_j \| ^2 _b b(z), \nn \ea
  so that
  \be \label{a0-calc-1}
  (I-\sum_{j=1}^d X_j^*X_j)b =  b(0)+ b\sum_{j=1}^d \|b_j\|_b^2.
  \ee
(Here we have used the fact that the $X_j^*$ act by
  \begin{equation*}
    (X_j^*f)(z)= z_jf(z) -\langle f, b_j\rangle_b b(z),
  \end{equation*}
which follows easily from the definition of the $X_j$ and the reproducing formula $(X_j^*f)(z) =\langle X_j^*f, k_z^b\rangle_b = \langle f, X_jk_z^b\rangle_b$.)
Second, using the defect formula (\ref{eqn:defects}),
\begin{equation}\label{a0-calc-2}
    (I- \sum_{j=1}^d X_j^*X_j) b = b(0)k_0^b +|a_0|^2\|b\|_b^2 b = b(0)+(-|b(0)|^2 +|a_0|^2\|b\|_b^2)b(z).
  \end{equation}
Equating (\ref{a0-calc-1}) and (\ref{a0-calc-2}) gives
\begin{equation*}
b(0)+ b(z)\sum_{j=1}^d \|b_j\|_b^2 = b(0)+(-|b(0)|^2 +|a_0|^2\|b\|_b^2)b(z)
\end{equation*}
Subtracting $b(0)$ from both sides leaves an equality between two
constant multiples of $b(z)$; since $b$ is assumed nonzero we have
\begin{equation*}
-|b(0)|^2 +|a_0|^2\|b\|_b^2 = \sum_{j=1}^d \|b_j\|_b^2 = 1-|b(0)|^2-|a_0|^2
\end{equation*}
and solving for $|a_0|^2$ gives (\ref{a0-formula}).
\end{proof}

\section{The $a$-function}\label{sec:a-function}

In this section we prove the first half of Theorem~\ref{thm-main}:

\begin{prop}\label{prop-a-function} If $b$ is not quasi-extreme, then there exists a nonzero multiplier $a$ such that
  \begin{equation*}
    M_a^*M_a +M_b^*M_b\leq I.
  \end{equation*}
\end{prop}

In the one-variable case if $b$ is not extreme then there is an outer function $a$ defined by the property that
\begin{equation}\label{a-outer-def}
  |a(\zeta)|^2+|b(\zeta)|^2 =1 \quad \text{a.e. on }\mathbb T;
\end{equation}
we can assume that $a(0)> 0$. In the above $\T$ denotes the unit circle. For this $a$ we have immediately
$M_a^*M_a+M_b^*M_b=I$. It is known in general that an equality of this
sort cannot hold when $d>1$ except in trivial cases (where the
functions are constant), see \cite{guo-etal}. In any case, when $d>1$ we do
not have any direct recourse to the theory of outer functions so different
methods are required.

Nonetheless, the proof of (\ref{prop-a-function}) is in a sense constructive: $a$ will be given in terms of a {\em transfer function realization} \cite{BTV},\cite{BBF}. It is remarkable that the algebraic construction given here, if carried out in one variable, produces exactly the outer function in (\ref{a-outer-def}). This follows from our transfer function realization and Sarason's computation of the Taylor coefficients of $a$ \cite{Sar}; we prove this at the end of the section.

We begin by recalling the relevant facts about transfer function realizations \cite{BTV} and the {\em generalized functional models} of \cite{BBF}.

Let $\mathcal X,\mathcal U, \mathcal Y$ be Hilbert spaces and let
$\mathcal X^d$ denote the direct sum of $d$ copies of $\mathcal X$. By
a {\em $d$-colligation} we mean an operator ${\bf U}:\mathcal X\oplus
\mathcal U\to \mathcal X^d\oplus \mathcal Y$ expressed in the block matrix form
\begin{equation*}
  {\bf U} = \begin{bmatrix} A & B \\ C & D \end{bmatrix}
  = \begin{bmatrix} A_1 & B_1 \\ \vdots & \vdots \\ A_d & B_d \\ C &
    D\end{bmatrix} :\begin{bmatrix} \mathcal X \\ \mathcal U\end{bmatrix} \to \begin{bmatrix} \mathcal X^d \\ \mathcal Y\end{bmatrix}
\end{equation*}
The colligation is called {\em contractive, isometric, unitary, etc.} if ${\bf U}$ is an operator of that type. For points $z=(z_1, \dots z_d)\in\mathbb C^d$, it will be convenient to identify $z$ with the row contraction:
\begin{equation*}
  z :\mathcal X^d\to \mathcal X; \quad \quad z \bbm x_1 \\ \vdots \\ x_d \ebm := z_1 x_1 + ... + z_d x_d.
\end{equation*}
Observe that $\| z \|^2 = \| z z^* \| _{\L ( \mc{X})}  =\sum_{j=1}^d |z_j|^2$, so $\| z \| = |z | _{\C ^d} <1$ if and only if $z\in \mathbb B^d$. If ${\bf U}$ is a contractive colligation, the {\em transfer function} for ${\bf U}$ is
\begin{equation*}
  S(z) = D+ C(I-zA)^{-1}zB.
\end{equation*}
The transfer function $S(z)$ is a holomorphic function in $\mathbb
B^d$ taking values in the space of bounded operators from $\mathcal U$
to $\mathcal Y$. (For our purposes we will only need to consider
finite-dimensional $\mathcal U$ and $\mathcal Y$). It is a theorem of
Ball, Trent and Vinnikov \cite{BTV} that $b$ is a contractive
multiplier of $H^2_d\otimes \mathcal U$ into $H^2_d \otimes \mathcal
Y$ if and only if it possesses a transfer function realization. In
\cite{BBF}, it was shown that such a transfer function could always be
chosen to be of a special form, called a {\em generalized functional
  model} realization.  In particular, (in the case $\mathcal
U=\mathcal Y=\mathbb C)$ if $X=(X_1, \dots X_d)$ is a contractive
solution to the Gleason problem in $\mathcal H(b)$, if we take
$\mathcal X=\mathcal H(b)$ and define for all $f\in\mathcal H(b)$ and
$\lambda \in\mathbb C$
\begin{itemize}
\item $A_jf=X_jf$, $j=1, \dots d$
\item $B_j\lambda = \lambda b_j$, $j=1, \dots d$
\item $Cf =f(0)$
\item $D\lambda = b(0)\lambda$
\end{itemize}
then the corresponding colligation is contractive and its transfer
function is $b(z)$.  Since $Cf =f(0)=\langle f, k_0^b \rangle$, we will write $k_0^{b^*}$ for $C$ and express the colligation as
\begin{equation*}
  {\bf U} = \begin{bmatrix} X_1 & b_1 \\ \vdots & \vdots \\ X_d & b_d \\ k_0^{b*} & b(0)\end{bmatrix}
\end{equation*}

\begin{proof}[Proof of Proposition~\ref{prop-a-function}]
Fix the admissible tuple $(b_1, \dots b_d)$ and corresponding operators $(X_1, \dots X_d)$ of  Proposition~\ref{canonical-solution}.
We can then consider the colligation acting between $\mathcal H(b)\oplus \mathbb C$ and $\mathcal H(b)\oplus \mathbb C^2$ given by
\begin{equation*}
{\bf \widetilde{U}} = \begin{bmatrix} X_1 & b_1 \\ \vdots & \vdots \\ X_d & b_d \\ k_0^{b*} & b(0) \\ -a_0b^* & a_0\end{bmatrix}
\end{equation*}

We claim that ${\bf \widetilde{U}}$ is isometric. If this is so, then
the colligation
\begin{equation*}
{\bf V} = \begin{bmatrix} X_1 & b_1 \\ \vdots & \vdots \\ X_d & b_d \\ -a_0b^* & a_0\end{bmatrix}
\end{equation*}
is contractive, and hence the associated transfer function is a
contractive multiplier $a(z)$. Moreover it is apparent that $a$ is nonzero, since $a(0)=a_0\neq 0$.

With $a$ defined in this way, the transfer function associated to ${\bf \widetilde{U}}$ is the $2\times 1$ multiplier
\begin{equation*}
  \begin{pmatrix} b \\ a\end{pmatrix}
\end{equation*}
which is contractive; this proves the proposition.

It remains to prove the claim that ${\bf \widetilde{U}}$ is isometric. Let us write out ${\bf \widetilde{U}}^* {\bf \widetilde{U}}$ explictly; we have
\begin{equation*}
  {\bf \widetilde{U}}^* {\bf \widetilde{U}} = \begin{bmatrix} \sum_{j=1}^d X_jX_j^* + k_0^b\otimes k_0^b +|a_0|^2b\otimes b & \sum_{j=1}^d X_j^*b_j  +b(0)k_0^b -|a_0|^2 b\\ \ast & \sum_{j=1}^d \|b_j\|_b^2 +|b(0)|^2+|a_0|^2\\\end{bmatrix}
\end{equation*}
where we note that $(2,1)$ entry is just the adjoint of the $(1,2)$ entry.  We consider the entries of the right-hand side one at a time.

The $(1,1)$ entry is equal to the identity operator on $\mathcal H(b)$ by (\ref{eqn:defects}).

The $(2,2)$ entry is equal to $1$ by the definition of $a_0$ in Proposition~\ref{defect-identity}.

The $(1,2)$ (and by symmetry, $(2,1)$) entry is equal to $0$. To see this we use again the fact that $X_jb=b_j$ and compute:
\begin{align*}
  \sum_{j=1}^d X_j^*b_j &= \sum_{j=1}^d \left[z_jb_j(z) -b(z)\|b_j\|_b^2 \right] \\
&= -b(0) +b(z)(1-\sum_{j=1}^d \|b_j\|_b^2) \\
&= -b(0) +(|a_0|^2 +|b(0)|^2)b(z) \\
&= -b(0)(1-b(0)^*b(z)) +|a_0|^2b(z) \\
&= -b(0)k_0^b + |a_0|^2b(z)
\end{align*}
Thus ${\bf \widetilde{U}}$ is isometric, which finishes the proof.

\end{proof}

\subsection{The one-variable case}

We analyze the foregoing construction in the one-variable
case. Here the Drury-Arveson space becomes the classical Hardy
space $H^2(\mathbb D)$ and its multiplier algebra is the space of
bounded analytic functions $H^\infty(\mathbb D)$, equipped with the
supremum norm. In this case it is known \cite{Sar-book} that
$b\in ball(H^\infty)$ is quasi-extreme if and only if it is an extreme
point of $ball(H^\infty)$, which is equivalent to the condition
\begin{equation}\label{szego-integral}
  \int_{\mathbb T} \log (1-|b|^2)\, dm =-\infty.
\end{equation}
(See \cite[p.138]{Hof-book}). Conversely, if $b$ is not (quasi)-extreme, this integral is finite,
and hence there exists (as noted at the beginning of this section) an outer function $a\in ball(H^\infty)$
satisfying
\begin{equation}\label{a-disk-def}
  |a(\zeta)|^2+|b(\zeta)|^2 = 1
\end{equation}
for almost every $|\zeta|=1$; this $a$ is unique if we impose the
normalization $a(0)>0$.

In this setting, there is of course ever only one solution to the
Gleason problem in $\mathcal H(b)$, namely the usual backward shift
operator on holomorphic functions
\begin{equation*}
  S^*f(z)= \frac{f(z)-f(0)}{z}.
\end{equation*}
Following Sarason \cite{Sar} we denote the restriction
\begin{equation*}
  X =S^*|_b.
\end{equation*}
All of the above discussion of transfer function realizations applies
here, so $b$ is realized by the colligation
\begin{equation*}
  {\bf U} = \begin{bmatrix} X & S^*b \\ k_0^{b*} & b(0)\end{bmatrix}
\end{equation*}

Let now $a$ be the outer function of (\ref{a-disk-def}) with $a(0)>0$.

We expand $a$ as a power series
\begin{equation*}
  a(z)=\sum_{n=0}^\infty \hat{a}(n)z^n.
\end{equation*}
Sarason \cite{Sar}  proves the following formula for the Taylor coefficients
$\hat{a}(n)$:
\begin{prop}\label{a-coefficients}
  We have $|a(0)|^2 = \frac{1}{1+\|b\|_b^2}$ and for $n\geq 1$
  \begin{equation*}
    \langle X^nb,b\rangle_{{\mathcal H}(b)} = \frac{-\hat{a}(n)}{a(0)}.
  \end{equation*}
\end{prop}

Mutliplying by $z^n$ and
summing we get
\begin{align*}
  a(z) &= a(0) -a(0)\sum_{n=1}^\infty \langle X^n b,b\rangle_b z^n \\
&= a(0)-a(0)\langle (I-zX)^{-1} zXb, b\rangle_b
\end{align*}
Since $Xb=S^*b$, this shows that $a$ is a transfer function for the colligation
\begin{equation*}
  \begin{bmatrix}
X & S^*b \\ -a(0)b^*  &a(0) \end{bmatrix}
\end{equation*}
acting on $\mathcal H(b)  \oplus \mathbb C$.  Finally, since
\begin{equation*}
  a(0) = \sqrt{\frac{1}{1+\|b\|^2_b}} = a_0
\end{equation*}
(the first equality by Proposition~\ref{a-coefficients} and the second by
Lemma~\ref{a0-calculation}) this is precisely the transfer function
which is used to define $a$ in the proof of Proposition~\ref{prop-a-function}.

\section{Conclusion of the Proof of Theorem~\ref{thm-main}}\label{sec:forward}

In this section we prove the second half of Theorem~\ref{thm-main}:

\begin{prop}\label{minorant-implies-notqe} If $b$ is a multiplier of $H^2_d$ and there exists a nonzero multiplier $a$ such that
  \begin{equation*}
M_a^*M_a +M_b^*M_b\leq I
  \end{equation*}
then $b$ is not quasi-extreme.
\end{prop}

The proof requires an elementary-seeming lemma, which nonetheless
appears easiest to prove using the notion of a {\em free lifting} of a
multiplier. We review the relevant results, prove the lemma, and
finally prove Proposition~\ref{minorant-implies-notqe}.

We recall quickly the construction of the {\em free or non-commutative Toeplitz algebra} of Popescu.
This is a canonical example of a \emph{free semigroup algebra}
as described by Davidson and Pitts \cite{DavPit}, which contains proofs of
all the claims made here. Fix an alphabet of $d$ letters $\{1, \dots d\}$ and let $\mathbb
F_d^+$ denote the set of all words $w$ in these $d$ letters, including
the empty word $\varnothing$. The set $\mathbb
F_d^+$ is a saemigroup under concatenation: if $w=i_1, \dots i_n$ and
$v=j_1\dots j_m$, we define
\begin{equation*}
  wv=i_1\cdots i_nj_1\dots j_m.
\end{equation*}
Let $F^2_d$ denote the Hilbert space (called the {\em Fock space}) with orthonormal basis
$\{\xi_w\}_{w\in \mathbb F_d^+}$.  This space comes equipped with a
system of isometric operators $L_1, \dots L_d$ which act on basis
vectors $\xi_w$ by {\em left creation}:
\begin{equation*}
  L_i\xi_w = \xi_{iw}.
\end{equation*}
The operators $L_1, \dots L_d$ obey the relations
\begin{equation*}
  L_i^*L_j=\delta_{ji}I,
\end{equation*}
in other words they are isometric with orthogonal ranges. The {\em
  free semigroup algebra} $\mathcal L_d$ is the WOT-closed algebra of
bounded operators on $F^2_d$ generated by $L_1, \dots L_d$. Each
operator $F\in\mathcal L_d$ has Fourier-like expansion
\begin{equation}
  F\sim \sum_{w\in\mathbb F^+_d} f_wL^w
\end{equation}
where, for a word $w=i_1\cdots i_n$, by $L^w$ we mean the product
$L_{i_1}L_{i_2}\cdots L_{i_n}$. The coefficients $f_w$ are determined
by the relation
\begin{equation*}
  f_w =\langle F\xi_\varnothing, \xi_w\rangle_{F^2_d}
\end{equation*}
and the Cesaro means of the series converge WOT to $F$.  To each $F\in
\mathcal L_d$ we can associated a $d$-variable holomorphic function
$\lambda(F)$ as follows: to each word $w=i_1\cdots i_n$ let $z^w$ denote the product
\begin{equation*}
  z^w = z_{i_1}z_{i_2}\cdots z_{i_n}.
\end{equation*}
(Observe that $z^w=z^v$ precisely when $w$ is obtained by permuting
the letters of $v$). Then for $F\in \mathcal L_d$ we define
$\lambda(F)$ by the series
\begin{equation*}
  \lambda(F)(z)= \sum_{w\in\mathbb F^+_d} f_wz^w.
\end{equation*}
The series converges uniformly on compact subsets of $\mathbb B^d$,
and is always a multiplier of $H^2_d$. In fact, Davidson and Pitts
prove that the map $\lambda$ is completely contractive from $\mathcal
L_d$ to $\mathcal{M}(H^2_d)$. Conversely, if $f\in \mathcal{M}(H^2_d)$ and $\|f\|\leq
1$, then there exists (by commutant lifting) an $F\in \mathcal L_d$ (not necessarily unique)
such that $\|F\|\leq 1$ and $\lambda(F)=f$. We call such an $F$ a {\em
  free lifting} of $f$.   Free liftings also always exist for
matrix-valued multipliers, so in particular if, say,
\begin{equation*}
  \begin{pmatrix}f \\g\end{pmatrix}
\end{equation*}
is a contractive $2\times 1$ multiplier, then there exist
$F,G\in\mathcal L_d$ such that $\lambda(F)=f, \lambda(G)=g$, and
\begin{equation*}
  \begin{pmatrix} F\\G\end{pmatrix}
\end{equation*}
is contractive.

We will need the following lemma, which we prove using free liftings:
\begin{lem} If $b$ is a multiplier and there exists a nonzero multiplier $a$ satisfying $M_a^*M_a+M_b^*M_b\leq I$, then an $a$ can be chosen satisfying this inequality and such that $a(0)\neq 0$.
\end{lem}

\begin{proof}
 By the above remarks there exist free liftings $A, B$ of $a$ and $b$ to the free semigroup algebra $\mathcal L_d$ such that the column $\begin{pmatrix} B\\A\end{pmatrix}$ is contractive. The element $A$ has Fourier expansion
  \begin{equation*}
    A\sim\sum a_w L^w
  \end{equation*}
with $a_\varnothing=0$ (since $a(0)=0$). Choose a word $v$ of minimal length such that $c_v\neq 0$. It
follows that
\begin{equation*}
  \widetilde{A} = L_v^*A = \sum_w c_w L_v^*L_w =\sum_u \widetilde{c}_u L_u
\end{equation*}
is a contractive free multiplier, and
$\widetilde{A}(0):=\widetilde{c}_\varnothing = c_v\neq 0$, and we then
have that
\begin{equation*}
  \begin{pmatrix} B \\ \widetilde{A}\end{pmatrix} = \begin{pmatrix} I
    & 0 \\ 0 & L_v^*\end{pmatrix} \begin{pmatrix} B \\ A\end{pmatrix}
\end{equation*}
is contractive. Since the Davidson-Pitts symmetrization map $\lambda$ is
completely contractive, on putting
$\widetilde{a}=\lambda(\widetilde{A})$ we have $\widetilde{a}(0)\neq
0$ and
\begin{equation*}
  \begin{pmatrix} b \\ \widetilde{a}\end{pmatrix}
\end{equation*}
is a contractive $2\times 1$ multiplier, which proves the lemma.
\end{proof}
{\bf Remark:} This is really the same proof that works in the disk
(without the need for the free lifting step). In the disk we just get
that $\widetilde{a}$ satisfies  $a(z)=z^n\widetilde{a}(z)$ for some $n$, and hence
\begin{equation*}
  M_a^*M_a = M_{\widetilde{a}}^* M_{\widetilde{a}};
\end{equation*}
(since $M_z$ is an isometry).  More generally we could let $a=\theta F$ be the inner-outer factorization of $a$; since $M_\theta$ is isometric we would have
\begin{equation*}
  M_a^*M_a = M_F^* M_\theta^* M_\theta M_f =M_F^*M_F.
\end{equation*}

  \begin{proof}[Proof of Proposition~\ref{minorant-implies-notqe}]
    Suppose that $b$ is a contractive multiplier and there exists a
    nonzero multiplier $a$ so that
    \begin{equation*}
      M_a^*M_a+M_b^*M_b\leq I
    \end{equation*}
By the lemma we may assume that $a(0)\neq 0$. We will construct an admissible tuple $b_1, \dots b_d$ such
that
\begin{equation*}
  \sum_{j=1}^d \|b_j\|_b^2 \leq 1-|b(0)|^2-|a(0)|^2 < 1-|b(0)|^2;
\end{equation*}
by the remark following Proposition~\ref{prop:qe-conditions} this proves that $b$ is not quasi-extreme.

  Let
  \begin{equation*}
    c=\begin{pmatrix} b  & 0 \\ a & 0\end{pmatrix}.
  \end{equation*}
Then $c$ is a $2\times 2$ contractive multiplier, and
\begin{equation}
  c(0)^*c(0) = \begin{pmatrix} |b(0)|^2+|a(0)|^2 & 0 \\ 0 & 0 \end{pmatrix}.
\end{equation}

We form the deBranges-Rovnyak space $\mathscr H(c)$ of the function $c$,
which has reproducing kernel
\begin{align*}  k^c (z,w) & =  \frac{ I - c(z) c(w) ^* } {1-zw^*}  \\
& =  \begin{bmatrix} k^b (z,w) & \frac{-b(z)a(w) ^* }{1-zw^*} \\ \frac{-a(z) b(w) ^* }{1-zw^*} & k^a (z,w) \end{bmatrix}. \label{ckernel} \end{align*}

Now we apply the vector-valued generalization of a basic result from the theory of reproducing kernel Hilbert spaces:  let $H (k)$ be a $\mathcal{H}$-valued RKHS of functions on a set $X$.  An $\mathcal{H}$-valued function $F$ on $X$ belongs to $H (k) $ if and only if there is a $t \geq 0$ such that
$$ F(x) F(y) ^* \leq t^2 k (x,y), $$ as positive $\L (\mathcal{H} )$-valued kernel functions on $X$. Moreover
the least such $t$ that works is $t = \| F \| _{H (k)}$ \cite[Theorem 10.17]{Paulsen-rkhs}.

Note that in the above we view $F (x) : \C \rightarrow \mathcal{H}$ as a linear map for any fixed $x \in X$. It follows that $F(y)^* h = \ip{F(y)}{h} _\mathcal{H}$ for any $h \in \mathcal{H}$.
For example, if (as in the case of ${\mathscr H(c)}$) $\mathcal{H} = \C ^2$ then in the standard basis $ F(x) = \begin{bmatrix} F_1 (x) \\ F_2 (x) \end{bmatrix}$ and
$$ F(x)  F(y) ^* = \begin{bmatrix} F_1 (x) \ov{F_1 (y)} & F_1 (x) \ov{F_2 (y)} \\ F_2 (x) \ov{F_1 (y)} & F_2 (x) \ov{F_2 (y)} \end{bmatrix}. $$

So now let $C : \C ^2 \rightarrow K (c) \otimes \C^d $ be a contractive Gleason solution for $c$, e.g., the one appearing in a generalized functional model realization for $c$ (which exists by \cite{BBF}): that is, $C = \begin{bmatrix} c_1 \\ \vdots \\ c_d \end{bmatrix} $ obeys
$$ z C (z) = z_1 c_1 (z) + ... + z_d c_d (z) = c(z) - c(0), $$ and contractivity means:
$$ C ^* C \leq I - c(0) ^* c(0). $$ So each $c_j (z) \in \C ^{2\times 2}$ and we write
$$ c_j (z) = \begin{bmatrix} b_j (z) & * \\ a_j (z) & * \end{bmatrix}, $$ and observe that the $B = \begin{bmatrix} b_1 \\ \vdots \\ b_d \end{bmatrix}$,
and the $A = \begin{bmatrix} a_1 \\ \vdots \\ a_d \end{bmatrix}$ are Gleason solutions for $b , a $ in the sense that
$$ b(z) - b(0) = \sum_{j=1}^d z_j b_j (z), $$ and similarly for $a$. Note that
$$ c_j (z) e_1 = \begin{bmatrix} b_j (z) \\ a_j (z) \end{bmatrix}.$$  We need to check that $B$ actually belongs to ${\mathcal H(b)} \otimes \C ^d$ and is a contractive Gleason solution for $b$:
Let $\{ e_1 , e_2 \}$ denote the standard orthonormal basis of $\C ^2$ and let $t_j := \| c _k e_1 \| _{{\mathscr H(c)}}.$
Then by the vector-valued RKHS proposition discussed above, and the form of the reproducing kernel for ${\mathscr H(c)}$,
\begin{align*} (c_j (z) e_1)(c_j (w) e_1 ) ^* & =  \begin{bmatrix} b_j (z) \\ a_j (z) \end{bmatrix} \bbm b_j (w) ^* & a_j (w) ^* \ebm  \\
& =  \begin{bmatrix} b_j (z) b_j(w)^* & b_j (z) a_j (w)^* \\ a_j (z) b_j (w) ^* & a_j (z) a_j (w) ^* \end{bmatrix}  \\
& \leq   t_j ^2 \begin{bmatrix} k^b (z,w) & \frac{-b(z)a(w) ^* }{1-zw^*} \\ \frac{-a(z) b(w) ^* }{1-zw^*} & k^a (z,w) \end{bmatrix}, \end{align*}
as positive kernel functions. In particular the $(1,1)$ entry of the above equation must be a positive kernel function so that
$$ b_j (z) b_j(w)^* \leq t_j ^2 k^b (z,w). $$ Again, by the scalar version of the RKHS result this implies that $b_j \in {\mathcal H(b)}$ and that
$$ \| b_j \| _{{\mathcal H(b)}} \leq t_J = \| c_j e_1 \| _{{\mathscr H(c)}}. $$

This yields the inequalities
\begin{align*} \sum _{k=1} ^d \| b_j \| ^2 _{{\mathcal H(b)}} & \leq  \sum _{k=1} ^d t_j ^2  \\
& =  \sum _{k=1} ^d \| c_j e_1 \| ^2 _{{\mathscr H(c)}}  \\
& =  \sum _{k=1} ^d \ip{c_j ^* c_j e_1}{e_1} _{\C ^2 }  \\
& =  \ip{C^* C e_1}{e_1} _{\C ^2}  \\
& \leq  \ip{ (I - c(0) ^* c(0) ) e_1}{e_1} _{\C ^2} \\
& =  1 - |b(0) | ^2 - |a (0) | ^2, \\
& <   1 - |b (0) | ^2 ,  \end{align*} and the proof is complete.
\end{proof}

\bibliographystyle{plain}
\bibliography{QE}
\end{document}